\documentclass[12pt]{amsart}
\usepackage{ amssymb }
\usepackage{a4wide,xcolor}
\usepackage[alphabetic]{amsrefs}
\usepackage[T1,T2A]{fontenc}
\usepackage[utf8]{inputenc}
\usepackage{url,hyperref}
\usepackage{amsfonts}
\usepackage{mathtools}
\usepackage{needspace}
\usepackage[pdftex]{graphicx}
\usepackage{datetime}
\usepackage{epigraph}
\usepackage{verbatim}
\usepackage{mathtools}
\usepackage{enumerate}
\usepackage[american]{babel}
\numberwithin{equation}{section}

\newcommand{\Z}{\mathbb{Z}}

\newcommand{\N}{\mathbb{N}}
\newcommand{\R}{\mathbb{R}}

\newcommand{\Cm}{\mathbb{C}}

\newcommand{\eps}{\varepsilon}

\DeclareMathOperator{\supp}{supp} 

\newcommand{\Fo}{\mathfrak{F}}
\newcommand{\B}{\mathcal{B}}
\newcommand{\F}{\mathcal{F}}

\renewcommand{\phi}{\varphi}

\newtheorem{Thm}{Theorem}[section]
\newtheorem{theorem}[Thm]{Theorem}

\newtheorem{lemma}[Thm]{Lemma}
\newtheorem{proposition}[Thm]{Proposition}

\newtheorem{remark}[Thm]{Remark}

\begin{document}
  %\sloppy 

  \title[Sharp estimates for eigenvalues of localization operators before the plunge region]
  {Sharp estimates for eigenvalues of localization operators before the plunge region}
\author{Aleksei Kulikov}
\address{Aleksei Kulikov,
\newline University of Copenhagen, Department of Mathematical Sciences,
Universitetsparken 5, 2100 Copenhagen, Denmark,
\newline {\tt lyosha.kulikov@mail.ru} 
}
%\author{Fedor Nazarov}  
 %\address{Fedor Nazarov,
%\newline Department of Mathematical Sciences, Kent State University, Kent
%OH 44242, USA,
%\newline {\tt nazarov@math.kent.edu} }

  %\subjclass[2020]{Primary 30C40; Secondary 81R30, 49Q10,49R05}
  %\keywords{Uncertainty principle, Wehrl entropy, Faber--Krahn inequality, shape
%optimization}

  \begin{abstract} We study  two closely related yet different localization operators: the time-frequency localization operator to the pair of intervals $S_{I, J} = P_I \mathcal{F}^{-1} P_J\mathcal{F} P_I$ and the localization of the coherent state transform to the square $L_Q$. Eigenvalues of both of them exhibit the same phase transition: if $|I| |J| = |Q| = c$ then first $\approx c$ eigenvalues are very close to $1$, then there are $o(c)$ intermediate eigenvalues and the rest of the eigenvalues are very close to $0$. Moreover, for both of them if $n < (1-\eps)c$ for fixed $\eps > 0$ then the eigenvalues are exponentially close to $1$. The goal of this paper is to establish sharp uniform bounds on these eigenvalues when $n$ is close to $c$ and see if there is a qualitative difference between the spectrums of $S_{I, J}$ and $S_Q$. We show that for $n < c -c^{0.99}$, say, in the time-frequency localization case  we have  $-\log(1-\lambda_n(c))\asymp\frac{c-n}{\log(\frac{2c}{c-n})}$ while in the coherent state transform case we have  $-\log(1-\mu_n(c))\asymp (\sqrt{c}-\sqrt{n})^2,$ which is much smaller if $c-n = o(c)$, so there is indeed a difference between these two cases. The proofs crucially rely on the complex-analytic interpretations of these localization operators.
\end{abstract}
  \maketitle
\pagestyle{plain}
  \section{Introduction and main results}
\subsection{Time-frequency localization operator}
For a measurable set $A\subset\R^d$ by $P_A$ we denote the multiplication operator by $\chi_A$, and for a pair of sets $A, B\subset \R^d$ we define the time-frequency localization operator associated with them by $S_{A, B} = P_A \F^{-1}P_B\F P_A$, where $\F$ is the Fourier transform
$$\F f(\xi) = \hat{f}(\xi)=\int_{\R^d}f(x) e^{-2\pi i \langle x, \xi\rangle}dx.$$

For any sets $A, B$ the operator $S_{A, B}$ is self-adjoint and non-negative definite. If $A$ and $B$ have finite measures then $S_{A, B}$ is a trace class operator with ${\rm tr}(S_{A,  B}) = |A| |B|$, in particular it is compact. As such, it has a sequence of eigenvalues $$1 > \lambda_1(A, B) \ge \lambda_2(A, B)\ge \ldots \ge 0.$$
Note that the first eigenvalue is strictly less than $1$ as there are no non-zero functions $f$ such that $f$ and $\hat{f}$ have supports of finite measure -- this is known as the uncertainty principle of Benedicks and Amrein--Berthier \cite{Benedicks, AB}. The study of spectral properties of $S_{A, B}$ was initiated in a series of papers by  Landau, Pollak and Slepian \cite{SPL1, SPL2, SPL3}. There also were recent uniform estimates on the distribution of eigenvalues of $S_{A, B}$ under various assumptions about the geometry of sets $A, B$ and their boundaries, see e.g. \cite{Mar}.
\medskip

In this paper we will focus on the most basic and well-studied case where $d = 1$ and $A, B$ are two intervals on the real line. It is not hard to see using a linear change of variables that if we shift any of the sets $A$ or $B$ or replace $(A, B)$ with $(rA, r^{-1}B)$ then the eigenvalues  $\lambda_n(A, B)$ will not change (eigenfunctions will of course change under these transformations). So, for the case of two intervals the eigenvalues will depend only on the product of their lengths $c = |A| |B|$ and so we have a sequence $$1 > \lambda_1(c) > \lambda_2(c) > \ldots > 0.$$
We will be interested in the regime when the parameter $c$ is large. It turns out that in this case the eigenvalues exhibit a phase transition: first $\approx c$ of them are very close to $1$ (yet strictly less than $1$ by the aformentioned uncertainty principle), then there are only $\asymp \log c$ eigenvalues of intermediate size which constitute the so-called plunge region, and the remaining eigenvalues decay to zero very fast.

The decay region is the one which is the most relevant to applications, and the behaviour of eigenvalues in it is understood quite well. In \cite{Wid} Widom established the precise decay of the eigenvalues for fixed $c$, showing in particular that $\lambda_n(c)$ decays superexponentially in $n$, and then there was a series of works on establishing uniform estimates for the eigenvalues in this region in particular by Osipov \cite{Osi} and Bonami, Jaming and Karoui \cite{Jam}, culminating in the work of Bonami and Karoui \cite{Bon} which gave an extremely precise approximation for the eigenvalues when $n \ge c$.
\begin{theorem}[{\cite[Theorem 1]{Bon}}]\label{Bonami}
For all $n \ge c \ge 10$ we have
$$\lambda_n(c) = \exp\left(-\frac{\pi^2\left(n+\frac{1}{2}\right)}{2} \int_{\Phi\left(\frac{c}{n+\frac{1}{2}}\right)}^1 \frac{1}{tE(t)^2}dt+O(\log(n))\right),$$
where $E(t) = \int_0^1 \sqrt{\frac{1-t^2x^2}{1-x^2}}dx$ is the elliptic integral of the second kind and $\Phi$ is the inverse of the function $t\to \frac{t}{E(t)}$. 	
\end{theorem}
\begin{remark}
In this and some of the following cited results precise statements of the theorems might differ slightly from the published versions, as they use different normalizations of the Fourier transform.
\end{remark}
In the plunge region the first basic estimate was discovered by Slepian \cite{Sle} and rigorously proved by Landau and Widom \cite{Lan}, giving one of the ways to quantify the claim that the width of the plunge region is logarithmic.
\begin{theorem}\label{Slepian}
For a fixed $b\in \R$ we have
$$\lim_{c\to \infty} \lambda_{n(c, b)}(c) = (1+e^b)^{-1},$$
where $n(c,b) = [c + \frac{1}{\pi^2}b\log(c)]$ and $[t]$ is the largest integer not greater than $t$.
\end{theorem}
%In particular for $b = 0$ we get that $\lambda_{[c]}(c)\to \frac{1}{2}$. %This was later improved by Landau \cite{} who showed that for $c > 10$ we have $\lambda_{[c]-1}(c) > \frac{1}{2} > \lambda_{[c]+1}(c)$. 

Since the dependence on $b$ in Theorem \ref{Slepian} is exponential, if we turn to the counting function for the set $\Lambda_\eps(c) = \{n: 1-\eps > \lambda_n(c) > \eps\}$ then the dependence on $\eps$ would be logarithmic. However, since in Theorem \ref{Slepian} $b$ is required to be fixed (and the proof in \cite{Lan} does not allow for varying $b$), the question of uniform upper bounds on $|\Lambda_\eps(c)|$ is of interest. In \cite{Isr} Israel showed that for any $\eta > 0$ there exists $A_\eta$ such that for all $c > 10$ and $\eps < \frac{1}{2}$ we have $|\Lambda_\eps(c)|\le A_\eta\log^{2+\eta}(c\eps^{-1})$. Recently, this was improved to a near-optimal upper bound by Karnik, Romberg and Davenport \cite{Kar} who proved the following result.
\begin{theorem}[{\cite[Theorem 3]{Kar}}]\label{Karnik}
For all $c > 0$ and $0 < \eps < \frac{1}{2}$ we have
$$|\Lambda_\eps(c)|\le \frac{2}{\pi^2} \log(50c+25)\log\left(\frac{5}{\eps(1-\eps)}\right)+7.$$
\end{theorem}
Note that the dependence on $c$ and on $\eps$ here is logarithmic, as expected, and that the constant $\frac{2}{\pi^2}$ is best possible, as can be seen from Theorem \ref{Slepian} (it is $\frac{2}{\pi^2}$ as opposed to $\frac{1}{\pi^2}$ because we have to cover large positive and large negative $b$ at the same time).

In this work we will study the behaviour of the eigenvalues $\lambda_{n}(c)$ before the plunge region, that is for $n < c$. For fixed $n$ Fuchs \cite{Fuc} showed that the eigenvalues satisfy 
$$1-\lambda_n(c) \sim 4\sqrt{2}\frac{(4\pi)^n}{n!} c^{n-\frac{1}{2}}e^{-\pi c},$$
in particular that they are exponentially close to $1$. This exponential closeness was extended to $n< 0.58 c$ by Bonami, Jaming and Karoui \cite{Jam} and then to $n < (1-\delta)c$ for any $\delta > 0$ by the author \cite{Kulikov} at the expense of the exponent becoming worse when $\delta\to 0$.
\begin{theorem}\label{Kulikov}
For any $\delta > 0$ there exists $\gamma = \gamma(\delta) < 1$ such that for all big enough $c$ and all $n < (1-\delta)c$ we have
$$\lambda_n(c) \ge 1 - \gamma^c.$$
\end{theorem}
The primary goal of the present work is to find out how the pure exponential closeness in Theorem \ref{Kulikov} transforms into the exponential closeness with an extra $\log (c)$ factor as in Theorem \ref{Slepian} and Theorem \ref{Karnik}. We show that the factor interpolating between them is $\log(\frac{2c}{c-n})$ which is constant if $n < (1-\delta)c$ and which behaves like $\log(c)$ when $n-c$ is small. Specifically, we prove the following two theorems which give respectively lower and upper bounds on the eigenvalues for $n < c$.
\iffalse
\begin{theorem}\label{lower time}
For all $c > 1000$ and $n < c$ we have
$$\lambda_n(c) > 1 - \exp\left(-0.01\frac{c-n}{\log(\frac{2c}{c-n})}\right).$$
\end{theorem}
\begin{theorem}\label{upper time}
For all $c > 1000$ and $n < c-\log^2(c)$ we have
$$\lambda_n(c) < 1 - \exp\left(-100\frac{c-n}{\log(\frac{2c}{c-n})}\right).$$
\end{theorem}
\fi
\begin{theorem}\label{lower time}
There exist numbers $c_0, \eta > 0$ such that for $c > c_0$ and $n < c$ we have
$$\lambda_n(c) > 1 - \exp\left(-\eta\frac{c-n}{\log(\frac{2c}{c-n})}\right).$$
\end{theorem}
\begin{theorem}\label{upper time}
There exist numbers $c_0, \kappa > 0$ such that for $c > c_0$ and $n < c-\log^2(c)$ we have
$$\lambda_n(c) < 1 - \exp\left(-\kappa\frac{c-n}{\log(\frac{2c}{c-n})}\right).$$
\end{theorem}
%{\color{red}We must mention that both of these results would not have been possible without the generous help of Fedor Nazarov. Specifically, he came up with the construction of the thinning out set of points by means of a dyadic decomposition for Theorem \ref{lower time} and the idea to take a linear combination of eigenfunctions vanishing on this set for Theorem \ref{upper time}. Naturally, all mistakes and inaccuracies in the presentation are solely on the author of the present text.}

\begin{remark}The condition $n < c - \log^2(c)$ corresponds to the fact that in many places in our proofs we have polynomial in $c$ losses, and so we require the exponential factor in the theorems to be at most some negative power of $c$, which means that we need $c - n$ to be at least some multiple of $\log^2(c)$. In fact, we prove Theorem \ref{lower time} only for $n < c - c^{7/8}$ and cover the remaining cases by combining Theorem \ref{Slepian} and Theorem \ref{Karnik}. We suspect, although can not prove, that $\log^2(c)$ term in Theorem \ref{upper time} can be removed as well.
\end{remark}
\begin{remark}
Recently, Martin Dam Larsen and the author \cite{Martin} proved Theorem \ref{lower time} by completely different, purely Fourier analytic methods in the course of their study of the distribution of eigenvalues $\lambda_n(A, B)$ for sets in arbitrary dimensions.
\end{remark}
\subsection{Coherent state transform} Another way to localize a function in time and frequency is to consider its coherent state transform, also known as a short-time Fourier transform in the signal processing literature. For this we fix an $L^2$-normalized window function $\phi:\R^n\to \Cm$ and consider the inner product of the function $f\in L^2(\R^n)$ with the time-frequency shifts of $\phi$:
$$V_\phi f(x, \omega) = \int_{\R^n} f(t)\overline{\phi_{x, \omega}(t)}dt,$$
where $\phi_{x, \omega}(t) = \phi(t-x)e^{2\pi i t \cdot \omega}$. For the applications of this object in physics see \cite{Kla}. In signal processing short-time Fourier transform was introduced by Daubechies \cite{Dau} as a way to capture the phase space picture of the function $f$. It is immediate to see via the Cauchy--Schwarz inequality that $|V_\phi f(x, \omega)|\le \|f\|_{L^2(\R^n)}$ and it is also known that $V_\phi f$ is a continuous function, that $\|V_\phi f\|_{L^2(\R^{2n})} = \|f\|_{L^2(\R^n)}$ and moreover that we have the following resolution of the identity
$$f = \int_{\R^{2n}} V_\phi f(x, \omega) \phi_{x, \omega}dxd\omega.$$
For the proofs of these facts and more thorough exposition of the accompanying theory see e.g. \cite[Chapter 3]{Gro}. Having this representation of the identity operator, to get a localization operator it is natural to restrict the domain over which we are integrating. Specifically, for a measurable set $Q\subset \R^{2n}$ we consider
$$L_{Q}f = \int_{Q} V_\phi f(x, \omega) \phi_{x, \omega}dxd\omega.$$
This is always a self-adjoint non-negative definite operator with $L_Q \le \rm Id$. However, unlike in the time-frequency localization operator case, $V_\phi L_{Q} f$ will in general be supported on the whole $\R^{2n}$.   

Note that $L_Q$ also depends on the window $\phi$. Throughout this paper, we will only consider the case of the one-dimensional Gaussian function $\phi(t) = 2^{1/4}e^{-\pi x^2}$ which was the case considered in \cite{Dau} and for which the most results are available, so we suppress the dependence on $\phi$. If $\Omega$ has finite measure then $L_Q$ is a trace class operator with ${\rm tr}(L_Q) = |Q|$ (see e.g. \cite[Proposition 3.1]{Ric}. In this paper it was also shown that among the sets of a given measure disks maximize the Hilbert--Schmidt norm of $L_Q$). Therefore, it is a compact operator and as such it has a sequence of eigenvalues
$$1 > \mu_1(Q) \ge \mu_2(Q) \ge \ldots \ge 0.$$
Here, just like in the time-frequency localization case, it can be shown that the first eigenvalue is strictly less than $1$ (for otherwise $V_\phi f$ would have to be $0$ outside of $Q$, which would contradict, for example, the uniqueness theorem for the complex-analytic interpretation of $V_\phi f$ which we will use later).

If the set $Q$ is a disk centred at the origin then the Hermite functions are the eigenfunctions of $L_Q$ (if $Q$ is radially symmetric then it is still true, but the order of eigenvalues is not as simple to determine, see \cite{Simon} for the study of this issue). In particular, very precise estimates for the eigenvalues are available in this case. But from the point of view of comparison with the time-frequency localization operator, the case of interest is when $Q$ is a square of area $c$. In this case the eigenvalues $\mu_n(c)$ exhibit a similar phase transition to the time-frequency localization operator case, with $\approx c$ eigenvalues being close to $1$, then $o(c)$ intermediate eigenvalues and then eigenvalues going to $0$ very fast. However, the width of the plunge region here is considerably larger, it is of order $\sqrt{c}$ as opposed to $\log(c)$, see \cite{Old} (there it is assumed that the domain is $C^2$, but the proof works for the square as well. Also, the assumption that $\phi$ is Gaussian is not used there, the argument works for many different domains and different windows).

All previous lower bounds on the eigenvalues $\lambda_n(c)$ apart from the general estimates for the width of the plunge region like Theorem \ref{Slepian} and Theorem \ref{Karnik} either implicitly or explicitly went through the operator $L_Q$ for some domain $Q$ (either a disk or a square), informally saying that $\mu_n(c)$ give a lower bound for $\lambda_n(c)$ when $n < c$. In particular, $\mu_n(c)$ are also exponentially close to $1$ when $n < (1-\eps)c$. So, the natural question arises whether $\mu_n(c)$ and $\lambda_n(c)$ are close when $n$ is close to $c$. We show that this is not the case and that when $c-n$ is small $1-\mu_n(c)$ is significantly larger than $1-\lambda_n(c)$.
\begin{theorem}\label{lower coherent}
There exist numbers $c_0, \eta > 0$ such that for all $c > c_0$ and $n < c-c^{0.99}$ we have
$$\mu_n(c) > 1 - \exp(-\eta (\sqrt{c}-\sqrt{n})^2).$$
\end{theorem}
\begin{theorem}\label{upper coherent}
There exist numbers $c_0, \kappa > 0$ such that for all $c > c_0$ and $n < c-\sqrt{c\log(c)}$ we have
$$\mu_n(c) < 1 - \exp(-\kappa (\sqrt{c}-\sqrt{n})^2).$$
\end{theorem}
In particular, Theorem \ref{upper coherent} and Theorem \ref{lower time} together show the asymptotic separation between $1-\mu_n(c)$ and $1-\lambda_n(c)$ if $c - n = o(c)$.
\begin{remark}The condition $n < c -\sqrt{c\log(c)}$ appears for the same reason as in Theorem \ref{upper time} -- we have some polynomial in $c$ losses and we have to compensate them. The reason for the $c^{0.99}$ term is our reliance on the disk packing from Lemma \ref{lem}, in particular our argument fundamentally does not allow us to have the power of $c$ to be $\frac{1}{2}$, or even $\frac{1}{2}+o(1)$, and so we did not try to optimize the power $0.99$.
\end{remark}
\subsection{Proof strategy} The proofs of Theorem \ref{lower time} and Theorem \ref{lower coherent} on one side and Theorem \ref{upper time} and Theorem \ref{upper coherent} on the other side, follow the same basic approach: we start with a simple functional-analytic way to bound the eigenvalues from an appropriate side, and then implement this in each respective situation, which takes the bulk of the work. 

For Theorem \ref{lower time} and Theorem \ref{lower coherent} we use the min-max characterization of the eigenvalues
$$\lambda_n(c) = \sup_{\dim V = n} \inf_{v\in V\backslash \{0\}} \frac{\|S_{I, J} v\|}{\|v\|}$$
and a similar formula for $L_Q$. To bound this expression from below, we have to exhibit a single subspace of dimension $n$ which gives the required lower bound. To construct a subspace of dimension $n$, the natural way is to take $n$ linearly independent vectors  $v_1, v_2, \ldots , v_n$ and let $V$ be their linear span. To get from here the lower bound on the $n$'th eigenvalue it is sufficient to have the following two conditions satisfied
\begin{enumerate}
\item For all $k = 1, 2, \ldots , n$ the ratio $\frac{\|S v_k\|}{\|v_k\|}$ is close to $1$,
\item The norm of any linear combination $\sum_{k=1}^n c_k v_k,\,\, c_k\in \Cm$ is not much smaller than the $\ell^1$-norm of the terms $\sum_{k=1}^n |c_k| \|v_k\|$.
\end{enumerate}
If we have these two conditions then we can show that for any $v = \sum_{k=1}^n c_k v_k$ the ratio $\frac{\| Sv\|}{\|v\|}$ is also close to $1$. Of these two conditions the second one is clearly harder to satisfy since it deals with all possible linear combinations and not just $n$ different numbers (although of course the first one is also necessary). So, the approach we would take is to first make sure that the second condition is satisfied and then see what bound we can get in the first condition.

The simplest way to satisfy the second condition is to demand that $v_k$ are pairwise orthogonal, in this case we even know the exact formula for $\|\sum_{k=1}^n  c_k v_k\|$. This was the approach taken in \cite{Jam} where they considered the Hermite functions and this led to the exponential bound for $n < 0.58 c$. However, we do not need precise orthogonality to achieve the second condition, it is enough to have almost orthogonality $\langle v_k, v_l\rangle \approx \delta_{k, l}$. This was done in \cite{Kulikov} where the time-frequency shifts of Hermite functions were considered to push this bound to $n < (1-\delta)c$ for any $\delta > 0$. It turns out that, with a lot of optimizations, the same method works for Theorem \ref{lower coherent}, but we have to more carefully choose the set of time-frequency shifts we consider (with the help of Lemma \ref{lem}), pass to a slightly smaller subsquare to achieve stronger concentration bounds outside the larger square, and exercise more care while establishing upper bounds on the inner products $\langle v_k, v_l\rangle$. Theorem \ref{lower coherent} is the only one of our four results whose proof is somewhat similar to a result previously appearing in the literature.

For Theorem \ref{lower time} the (almost) orthogonal sequence approach does not seem to be sufficient, so we have to come up with a different way to ensure the second condition. We will do this with a trick of biorthogonal sequences: if we have another sequence of vectors $w_k$ such that $\langle v_k, w_l\rangle = \delta_{k, l}$ then by the Cauchy--Schwarz inequality we can see that
$$\|\sum_{k=1}^n c_k v_k\| \ge \max_k \frac{|c_k|}{\|w_k\|}.$$
The natural setting for the time-frequency localization operator is the Paley--Wiener space, and in this space there are very well-known candidates for the biorthogonal vectors $w_k$ --- the reproducing kernels, inner products with which give the values of our function at the points. The simplest way would be to just take $v_k = \frac{\sin(\pi x)}{x-k}$, which is even an orthogonal sequence, but unfortunately it has very poor concentration and we will fail the first condition. We can modify this by considering $v_k = \frac{\sin(a^{-1}\pi x)}{x-ak}h(x-ak)$ for some $a > 1$ and the function $h$ being responsible for the decay. This sequence is biorthogonal to the system of reproducing kernels at the points $ak, k\in \Z$ and we have a parameter $a$ at our disposal (the closer $a$ is to $1$, the less decay can we guarantee from the function $h$). Unfortunately, even optimizing in $a$ the best this method could give is the same bound as in Theorem \ref{lower coherent}, which is nowhere near what we claim in Theorem \ref{lower time}.

The reason for this failure is that while the bound is optimal for points near the endpoints of the interval, we have a lot more concentration than we need near the centre of it. So, instead of the arithmetic progression $ak, k\in \Z$ we want to consider a  set which is more dense near the centre of the interval, and whose density drops when we approach the endpoints. The issue is that we can not easily find a sequence of functions vanishing at all these points, so the biorthogonality condition is hard to satisfy. The key idea which leads to the proof of Theorem \ref{lower time} is that we don't actually have to do it --- just like before, the almost biorthogonality condition $\langle v_k, w_l\rangle \approx \delta_{k, l}$ is sufficient for our purposes. Our set of points for the reproducing kernels will be a union of arithmetic progressions on dyadic subintervals of our interval, which get thinner the closer we are to the endpoints. Inside of each dyadic subinterval the sequences are biorthogonal, and between different subintervals we use the decay of the function $h$ to ensure that the inner products are small enough, even if they are not necessarily $0$. 

For Theorem \ref{upper time} and Theorem \ref{upper coherent} we will use a different fact from linear algebra: if we have a vector space $V$ of dimension $n$ and $n-1$ linear functionals on it, then we can find a non-zero vector $v\in V$ for which all of these functionals are $0$. So, we will take the linear span of the first $n$ eigenfunctions and take their linear combination which vanishes at these functionals (for which we still have the concentration by the min-max principle). For Theorem \ref{upper coherent} we will take as the linear functionals the inner products with the first $n-1$ eigenfunctions of the localization operator $L_{Q'}$ for a smaller square $Q'$. Then, from the simple trace bound the resulting vector $v$ will have a big portion of its $L^2$-mass concentrated in $Q\backslash Q'$ and will also have almost no $L^2$-mass in $\R^2\backslash Q$. We will derive a contradiction from this by using the subharmonicity of the logarithm of the associated analytic function (see Section $2$ for the details) if the $L^2$-mass in $\R^2\backslash Q$ is small enough.

For the situation in Theorem \ref{upper time} the best we were able to prove with this strategy (by taking a subinterval instead of a subsquare) is the upper bound of the form
$$\lambda_n(c) \le 1 -\exp(-\kappa (c-n))$$
which is weaker than the bound in Theorem \ref{upper time}. Instead, we will consider a vector $v$ which is orthogonal to the same reproducing kernels that we considered in the proof of Theorem \ref{lower time}. Here, just using the subharmonicity of the logarithm will be insufficient, we will use a more refined bound coming from the Jensen's formula for the value of the logarithm of an analytic function at a point which takes into account zeroes that it has and will give us the bound we claim in Theorem \ref{upper time}.

Lastly, let us mention that we will frequently and tacitly use the following property of the operators $S_{I, J}$ and $L_Q$: they both have the form $X^*X$ for a certain operator $X$, specifically $$S_{I, J} = (P_J\F P_I)^*(P_J\F P_I)$$
and 
$$L_Q = (P_Q V_\phi)^*(P_Q V_\phi).$$

The eigenvalues of the operator $X^* X$ are the same as the squares of the singular values of $X$, which can be computed with the same min-max procedure for the expression $\frac{\| Xv\|^2}{\|v\|^2}$.
\subsection{Structure of the paper} In Section $2$ we state the  facts about the Bargmann--Segal--Fock space and its relation with the coherent state transform. In Section $3$ we prove Theorem \ref{upper coherent}. In Section $4$ we prove Theorem \ref{lower coherent}. In Section $5$ we recall the basic properties of the Paley--Wiener space that we will use. In Section $6$ we prove Theorem \ref{upper time}. In Section $7$ we prove Theorem \ref{lower time}. Finally, in Section $8$ we state some remaining open questions regarding the eigenvalues $\lambda_n(c)$ and $\mu_n(c)$ for $n < c$ and what our methods can give for $n > c$.
\section{Bargmann transform}
For the function $\phi(x) = 2^{1/4}e^{-\pi x^2}$ we have
\begin{align*}
V_\phi f(x, \omega) = 2^{1/4}\int_\R f(t)e^{-\pi (t-x)^2} e^{- 2\pi i t \omega}dt=\\
2^{1/4}e^{-\pi i x\omega} e^{-\frac{\pi}{2}(x^2+\omega^2)}\int_\R f(t)e^{2\pi t(x-i\omega)-\frac{\pi}{2}(x-i\omega)^2-\pi t^2}dt.
\end{align*}
For a function $f$ we define its Bargmann transform at the point $z = x-i\omega\in \Cm$ as 
$$\B f(z) = 2^{1/4}\int_\R f(t) e^{2\pi t z - \frac{\pi}{2}z^2 - \pi t^2}dt.$$
By differentiation under the integral sign it is easy to see that this is a holomorphic function of $z$ and from the above computation we have $|V_\phi f(x, \omega)| = |\B f(x-i\omega)|e^{-\frac{\pi}{2}(x^2+\omega^2)}$. Therefore, Bargmann transform is an isometry from $L^2(\R)$ to the Bargmann--Segal--Fock space
$$\Fo = \{ F\in Hol(\Cm): \|F\|^2_{\Fo} = \int_\Cm |F(z)|^2e^{-\pi |z|^2}dz < \infty\},$$
where $dz$ is the Lebesgue measure on $\Cm$. It is in fact true that $\B$ is a bijection between $L^2(\R)$ and $\Fo$ (one way to see this is to show that Hermite function $H_k$ gets mapped to $\sqrt{\frac{\pi^k}{k!}}z^k$ which form an orthonormal basis for $\Fo$). For this and other facts about the Bargmann transform see e.g. \cite[Section 3.4]{Gro}.

Time-frequency shifts of the function $f$ correspond to shifts on the complex plane. They have an associated family of isometries on the space $\Fo$: for a function $F\in \Fo$ and $w\in \Cm$ the function
$$T_w F(z) = F(z-w)e^{\pi z\bar{w} - \frac{\pi}{2}|w|^2}$$
also belongs to $\Fo$, $\|F\|_\Fo = \|T_w F\|_\Fo$ and moreover this transformation simply shifts the distribution of $|F(z)|e^{-\frac{\pi}{2}|z|^2}$. From this invariance it is in particular easy to see that the eigenvalues $\mu_n(c)$ do not depend on which square of area $c$  we consider (with sides parallel to the real and imaginary axis, although it is also possible to show the invariance under rotations), so we will always consider the square centred at the origin.

Last fact that we need is the pointwise estimate for the values of the functions in $\Fo$, which follows from the pointwise estimates for the coherent transform: for all $F\in \Fo$ and $z\in\Cm$ we have $|F(z)|\le e^{\frac{\pi}{2}|z|^2}\|F\|_{\Fo}$.
\section{Proof of Theorem \ref{upper coherent}}\label{UPC}
Assume that for some big enough $c$ and for some $n < c - \sqrt{c\log(c)}$ we have $1-\mu_n(c)  \le \exp(-\kappa (\sqrt{c}-\sqrt{n})^2)$. Let us consider the square $Q'$ of area $n-1$ and the associated operator $L_{Q'}$ (if $n = 1$ then $Q' = \{0\}$). We will take a normalized linear combination $f$ of the first $n$ eigenfunctions of $L_Q$ which is orthogonal to the first $n-1$ eigenfunctions of $L_{Q'}$. Then by the min-max principle we have 
$$\int_{\R^2} |V_\phi f(x, \omega)|^2dxd\omega = 1,$$
$$\int_{Q'} |V_\phi f(x,\omega)|^2dxd\omega \le \mu_n(n-1),$$
$$\int_{Q} |V_\phi f(x,\omega)|^2dxd\omega \ge \mu_n(c)$$
and we want to arrive at a contradiction from here. We let $F = \B f$ which then satisfies
$$\int_\Cm |F(z)|^2e^{-\pi |z|^2}dz = 1,$$
$$\int_{Q'} |F(z)|^2e^{-\pi |z|^2}dz\le \mu_n(n-1),$$
$$\int_Q |F(z)|^2e^{-\pi |z|^2}dz \ge \mu_n(c).$$

First, we will need a very simple upper bound on $\mu_n(n-1)$: since ${\rm tr} L_{Q'} = n-1$, we have 
$$n-1=\sum_{k=1}^\infty \mu_k(n-1) \ge \sum_{k=1}^n \mu_k(n-1) \ge n\mu_n(n-1),$$
so $\mu_n(n-1) \le \frac{n-1}{n} < 1-\frac{1}{c}$. This means that
$$\int_{\Cm\backslash Q'} |F(z)|^2e^{-\pi |z|^2}dz > \frac{1}{c}.$$
We also have
$$\int_{\Cm \backslash Q} |F(z)|^2e^{-\pi |z|^2}dz \le 1-\mu_n(c).$$
Therefore,
$$\int_{Q\backslash Q'} |F(z)|^2e^{-\pi |z|^2}dz >\frac{1}{c} - \exp(-\kappa (\sqrt{c}-\sqrt{n})^2).$$
If $\kappa \ge 10$ and $c$ is big enough then the second term here is at most $c^{-2}$ (this is one of the places where we need $n <c -\sqrt{c\log(c)}$). So, if $c > 2$ we have
$$\int_{Q\backslash Q'} |F(z)|^2 e^{-\pi |z|^2}dz > \frac{1}{2c}.$$
This means that there exists $z_0\in Q\backslash Q'$ for which $|F(z_0)|^2 e^{-\pi |z_0|^2} > \frac{1}{2c |Q\backslash Q'|}$. We will bound $|Q\backslash Q'|$ from above simply by $|Q| = c$, so that $|F(z_0)|^2 e^{-\pi |z_0|^2} > \frac{1}{2c^2}$. Note that $Q\backslash Q'$ is a (not disjoint) union of four rectangles, one on the right, one on the left, one on the top and one on the bottom. We will assume that $z_0$ belongs to the right rectangle, the other $3$ cases are handled completely analogously. This means that ${\rm Re}\, z_0 \ge \frac{\sqrt{n-1}}{2}$.

Let $r = 10(\sqrt{c}-\sqrt{n})$ and consider the circle centred at $z_0$ of radius $r$. If $c$ is big enough, the points on the right $\frac{1}{100}$'th of this circle satisfy ${\rm Re}\, z \ge \frac{\sqrt{c}}{2}+1$ (in fact, a stronger bound $\frac{\sqrt{c}}{2} + \frac{\sqrt{\log(c)}}{10}$ is possible but not necessary for us). The function $\log |F(z)|$ is subharmonic so we have
\begin{equation}\label{subh1}
\log |F(z_0)| \le \int_0^1 \log |F(z_0 + re^{2\pi i t})|dt.
\end{equation}
For all but that $\frac{1}{100}$'th arc of the circle we will on the right-hand side use a pointwise bound $$|F(z_0 + re^{2\pi i t})| \le e^{\frac{\pi}{2}|z_0+re^{2\pi i t}|^2}.$$

For $w$ in the right $\frac{1}{100}$'th of the circle we are going to obtain a stronger bound on $|F(w)|$. Note that a disk of radius $1$ around $w$ is still outside of $Q$, so the integral over it is at most $1-\mu_n(c)$. Let us consider the function $G=T_w F$, we have $|G(0)| = |F(w)|e^{-\frac{\pi}{2}|w|^2}$ and 
$$\int_{|z|\le 1} |G(z)|^2e^{-\pi |z|^2}dz \le 1-\mu_n(c).$$
We have $e^{-\pi |z|^2} \ge e^{-\pi}$ for $|z| \le 1$ so
$$\int_{|z|\le 1}|G(z)|^2 dz\le \frac{1-\mu_n(c)}{e^{-\pi}}.$$
The function $G(z)$ is analytic, therefore $|G(z)|^2$ is subharmonic. So, 
$$|G(0)|^2 \le \frac{1}{\pi} \int_{|z|\le 1}|G(z)|^2dz \le \frac{1-\mu_n(c)}{\pi e^{\pi}}\le \frac{\exp(-\kappa (\sqrt{c}-\sqrt{n})^2)}{\pi e^{-\pi}}.$$
Taking the logarithm, dividing by $2$ and recalling that $|G(0)| = |F(w)|e^{-\frac{\pi}{2}|w|^2}$ we get
$$\log |F(w)| \le \frac{\pi}{2}|w|^2 - \frac{\kappa}{2}(\sqrt{c}-\sqrt{n})^2 +2.$$
Plugging our pointwise bounds into \eqref{subh1} we get
$$\log |F(z_0)| \le \int_0^1 \frac{\pi}{2} |z_0+re^{2\pi i t}|^2dt - \frac{1}{100} \left(\frac{\kappa}{2}(\sqrt{c}-\sqrt{n})^2-2\right).$$
The first integral is equal to $\frac{\pi}{2}|z_0|^2 + \frac{\pi}{2}r^2$. Combining this with our lower bound on $|F(z_0)|$ we get
$$\frac{\kappa}{200}(\sqrt{c}-\sqrt{n})^2 \le \frac{\pi}{2}r^2 + \log(c) + 10.$$
Recalling that $r = 10(\sqrt{c}-\sqrt{n})$ we get
$$(\sqrt{c}-\sqrt{n})^2\left(\frac{\kappa}{200}-50\pi\right) \le \log(c) +10$$
which is false if $\kappa$ is big enough (once more, here we need $c-n\ge \sqrt{c\log(c)}$ to beat $\log(c)$ term on the right-hand side).
\begin{remark}
We can try doing a similar argument in the time-frequency localization case, by taking a subinterval $J'=[-\frac{n-1}{2},\frac{n-1}{2}]$ of $J=[-\frac{c}{2},\frac{c}{2}]$ and considering a linear combination of the eigenfunctions of $S_{I,J}$ which is orthogonal to the first $n-1$ eigenfunctions of the  $S_{I,J'}$. Unfortunately, this only gives an upper bound $\lambda_n(c)\le 1 - \exp(-\delta(c-n))$ for some $\delta > 0$ as functions in the Paley--Wiener space can drop by only an exponential in $r$ factor over the course of an interval of length $r$.
\end{remark}
\section{Proof of Theorem \ref{lower coherent}}
We begin with the following disk packing lemma of Lieb and Lebowitz \cite{Lieb}.
\begin{lemma}[{\cite[Lemma 4.1]{Lieb}}]\label{lem}
There exists $\eps_0 > 0$ such that for all $0 < \eps < \eps_0$ there exists a finite set of pairwise disjoint disks $D_1, D_2, \ldots , D_N\subset (-\frac{1}{2}, \frac{1}{2})^2$ such that the measure of their union is at least $1-\eps$ and the radius of each of these disks is at least $\eps^{22}$. 
\end{lemma}
\begin{remark} In \cite{Lieb} the authors constructed such a family of disks explicitly and one can do simple algebra to check that their disks, if taken in the decreasing order of radiuses, give the claimed bound (in fact, their proof allows for any power larger than $\frac{\log(0.1)}{\log(0.9)}\approx 21.85$). On the other hand, some positive power of $\eps$ is unavoidable, because if the smallest radius is $r$ then the $\frac{r}{\sqrt{2}}$-vicinity of each of the corners is not covered by any disk, hence $r\le \sqrt{\eps}$.
\end{remark}
\begin{remark} A similar lemma was used in \cite{Kulikov}. However, there estimate on the smallest radius was not needed, simply the existence of any such disjoint finite union of disks was enough.
\end{remark}
As we mentioned in the introduction, we will obtain a lower bound on the eigenvalues $\mu_n(c)$ by the min-max principle. Since the Bargmann transform is an isometry from $L^2(\R)$ to $\Fo$ we will construct a subspace $V$ of $\Fo$ of dimension $n$ such that all functions $F\in V$ are mostly concentrated in the square $\left(-\frac{\sqrt{c}}{2},\frac{\sqrt{c}}{2}\right)^2$. The space that we will construct will be a linear span of the functions of the form $T_w\left(\sqrt{\frac{\pi^k}{k!}}z^k\right)$ for various $w\in\Cm, k\in\N_0$, which under the inverse Bargmann transform correspond to time-frequency shifts of Hermite functions, but we will not need this relation to Hermite functions in our proof.

Let $\theta = \sqrt[5]{\frac{n}{c}} < 1$ and $\eps = 1 - \theta > 0$. Note that $\eps\le\frac{c-n}{c}\le 5\eps$.  We can without loss of generality assume that $\eps < \eps_0$ from Lemma \ref{lem}. If $\eps \ge \eps_0$ we will pick the smallest $n_0$ such that $1-\sqrt[5]{\frac{n_0}{c}} < \eps_0$, the estimate for $n_0$ gives us that $1 > \mu_n(c) \ge \mu_{n_0}(c) > 1-\exp(-\delta c)$ for some absolute constant $\delta > 0$ depending only on $\eps_0$, so to get the estimate for $\mu_n(c)$ we can just decrease $\eta$ in Theorem \ref{lower coherent} to make it smaller than $\delta$.

We consider the set of disks provided by Lemma \ref{lem} for our $\eps$, so that their total area is at least $1-\eps = \theta$. Let disk $D_m$ have radius $r_m$ and centre $z_m\in \left(-\frac{1}{2},\frac{1}{2}\right)^2$ which we view as the complex number. We define a new family of disks $B_m$ with centres $w_m=\sqrt{c}\theta z_m$ and radiuses $R_m=\sqrt{c}\theta^2 r_m$. Note that since $\theta < 1$ these disks are still disjoint and their total area is at least $c\theta^5 = n$. Moreover, if $r$ is the smallest of $r_m$ (which we know is at least $\eps^{22}$) then the distance between any two of these disks is at least $2\sqrt{c}\theta r (1-\theta)$.

The set of functions that we will consider will be $$\mathcal{W}=\bigcup_{m=1}^N\left\{ T_{w_m}\sqrt{\frac{\pi^k}{k!}}z^k\mid 0 \le k \le \pi R_m^2\right\}.$$
Note that it has at least $n$ elements. We want to show that the functions from it have strong concentration to the square $\left(-\frac{\sqrt{c}}{2},\frac{\sqrt{c}}{2}\right)^2$ and that they are almost orthogonal. For this we will use the following estimate which essentially says that $\sqrt{\frac{\pi^k}{k!}}z^k$ has a Gaussian decay outside of a disk of area $k$.
\begin{proposition}
For all $k\in \N_0$ and all $R \ge 1$ we have
\begin{equation}\label{prop}
\int_{|z| > \sqrt{\frac{k}{\pi}}+R}\sqrt{\frac{\pi^k}{k!}}|z|^k e^{-\frac{\pi}{2}|z|^2}dz \le  (2+2\sqrt{k})e^{-\frac{\pi R^2}{2}}.
\end{equation}
\end{proposition}
\begin{proof}
Let us first assume that $k \ge 1$. We have for $r = \sqrt{\frac{k}{\pi}}+t$
\begin{align*}
k\log(r) - \frac{\pi}{2}r^2 = k\log\left(\sqrt{\frac{k}{\pi}}\right) + k\log\left(1+\sqrt{\frac{\pi}{k}}t\right) -\frac{\pi}{2} \left(\frac{k}{\pi} + 2\sqrt{\frac{k}{\pi}}t + t^2\right)\le \\ k\log\left(\sqrt{\frac{k}{\pi}}\right)  - \frac{k}{2} - \frac{\pi t^2}{2},
\end{align*}
where we used inequality $\log(1+x)\le x$ in the second step. For the factor $\sqrt{\frac{\pi^k}{k!}}$ we use a well-known inequality $k! \ge \left(\frac{k}{e}\right)^k$ to get for $|z| = \sqrt{\frac{k}{\pi}} + t$
$$\sqrt{\frac{\pi^k}{k!}}|z|^k e^{-\frac{\pi}{2}|z|^2} \le e^{-\frac{\pi t^2}{2}}.$$
Note that this estimate is true for $k = 0$ as well, so from now on we assume $k \ge 0$. We consider the integral on the left-hand side of \eqref{prop} in polar coordinates with this upper bound to get
$$2\pi \int_{R}^\infty \left(t+\sqrt{\frac{k}{\pi}}\right)e^{-\frac{\pi t^2}{2}}dt \le 2\pi \left(1+ \sqrt{\frac{k}{\pi}}\right)\int_R^\infty te^{-\frac{\pi t^2}{2}}dt = \left(2 + 2\sqrt{\frac{k}{\pi}}\right)e^{-\frac{\pi R^2}{2}},$$ 
where in the first step we used an estimate $\sqrt{\frac{k}{\pi}} \le t\sqrt{\frac{k}{\pi}}$ which is valid since $t\ge R \ge 1$. To get the bound \eqref{prop} it remains to note that $\pi \ge 1$.
\end{proof}

Now, we turn to the estimation of the inner products and concentration of our functions. We begin with the former. Let us pick indices $1\le m, l\le N$ and degrees $0 \le k \le \pi R_m^2$, $0\le s \le \pi R_l^2$. We want to bound 
$$\int_\Cm \left(T_{w_m}\sqrt{\frac{\pi^k}{k!}}z^k\right)\left( \overline{T_{w_l}\sqrt{\frac{\pi^s}{s!}}z^s}\right)e^{-\pi |z|^2}dz.$$
If $l = m$ then this integral is equal to $1$ for $k = s$ and it is equal to $0$ for $k\neq s$, which can be easily seen by switching to polar coordinates. For $l\neq m$ we put the absolute values under the integral sign. Since the distance between the disks $B_m$ and $B_l$ is at least $2\sqrt{c}\theta r(1-\theta)$, we can bound this integral from above by
$$\int_{|z-w_m|>R_m + \sqrt{c}\theta r(1-\theta)} + \int_{|z-w_l| > R_l + \sqrt{c}\theta r(1-\theta)}.$$
For the first one we will use the pointwise bound $\left|T_{w_l}\sqrt{\frac{\pi^s}{s!}}z^s\right|\le e^{\frac{\pi}{2}|z|^2}$ while for the second one we use the pointwise bound $\left|T_{w_m}\sqrt{\frac{\pi^k}{k!}}z^k\right| \le e^{\frac{\pi}{2}|z|^2}$. We get for the first integral
$$\int_{|z-w_m|>R_m + \sqrt{c}\theta r(1-\theta)}\left|T_{w_m}\sqrt{\frac{\pi^k}{k!}}z^k\right| e^{-\frac{\pi}{2}|z|^2}dz$$
and a similar expression for the other integral. Since $T_{w_m}$ is an isometry preserving the distribution of $|F(z)|e^{-\frac{\pi}{2}|z|^2}$, this integral is equal to 
$$\int_{|z| > R_m + \sqrt{c}\theta r(1-\theta)} \sqrt{\frac{\pi^k}{k!}}|z|^k e^{-\frac{\pi}{2}|z|^2}dz.$$
Since $R_m \ge \sqrt{\frac{k}{\pi}}$, by \eqref{prop} this is at most $(2+2\sqrt{k})e^{-\frac{\pi}{2} c \theta^2 r^2(1-\theta)^2}$ (here we assume that $\sqrt{c}\theta r(1-\theta)\ge 1$ which is true for big enough $c$ under our assumption $n \le c - c^{0.99}$ and will be seen from the following reasoning). Note that all our disks are definitely inside of the square $\left(-\frac{\sqrt{c}}{2},\frac{\sqrt{c}}{2}\right)^2$, so their areas are at most $c$, hence $k\le c$. We will bound $(2+2\sqrt{k})$ from above by $\frac{c}{2}$ for big enough $c$. We can also assume that $\theta \ge 0.99$ by decreasing $\eps_0$ if necessary (since $1-\theta = \eps < \eps_0$), so that $\theta^2 > \frac{2}{\pi}$. Combining everything, recalling that $1-\theta = \eps$ and $r > \eps^{22}$ we get
$$\left|\int_\Cm \left(T_{w_m}\sqrt{\frac{\pi^k}{k!}}z^k\right)\left( \overline{T_{w_l}\sqrt{\frac{\pi^s}{s!}}z^s}\right)e^{-\pi |z|^2}dz\right| \le ce^{-c\eps^{46}}.$$
Recall that $c-n \ge c^{0.99}$. Therefore, $\eps \ge \frac{c^{-0.01}}{5}$. Using this bound we get
$$\left|\int_\Cm \left(T_{w_m}\sqrt{\frac{\pi^k}{k!}}z^k\right)\left( \overline{T_{w_l}\sqrt{\frac{\pi^s}{s!}}z^s}\right)e^{-\pi |z|^2}dz\right| \le ce^{-5^{-46}c^{1-0.46}} \le \frac{1}{2c}$$
for big enough $c$ (here we also get that $\sqrt{c}\theta r(1-\theta)\ge 1$ for big enough $c$). This is the only place in our argument where the power $0.99$ plays a role. In fact, any power larger than $1-\frac{1}{46}$ would have worked.

Now, we turn to the estimation of the concentration. We are interested in
$$\int_{\Cm \backslash\left(-\frac{\sqrt{c}}{2},\frac{\sqrt{c}}{2}\right)^2}\left|T_{w_m}\sqrt{\frac{\pi^k}{k!}}z^k\right|^2 e^{-\pi |z|^2}dz.$$
We notice that the disk $B_m$ is inside of the square $\left(-\frac{\theta\sqrt{c}}{2},\frac{\theta\sqrt{c}}{2}\right)^2$, so it is at least $\frac{(1-\theta)\sqrt{c}}{2}$ away from $\Cm \backslash\left(-\frac{\sqrt{c}}{2},\frac{\sqrt{c}}{2}\right)^2$. Using the bound $\left|T_{w_m}\sqrt{\frac{\pi^k}{k!}}z^k\right| \le e^{\frac{\pi}{2}|z|^2}$ to remove the square and $e^{-\frac{\pi}{2}|z|^2}$ we can once again invoke \eqref{prop} to get that this integral is at most $ce^{-\frac{\pi \eps^2 c}{8}}$. We notice that $$\frac{\pi \eps^2 c}{8} \ge \frac{(c-n)^2}{100c} \ge \frac{(\sqrt{c}-\sqrt{n})^2}{100}.$$  Since $n-c\ge c^{0.99}$, we have $$ce^{-\frac{(\sqrt{c}-\sqrt{n})^2}{100}} \ge \frac{1}{2c}e^{-\frac{(\sqrt{c}-\sqrt{n})^2}{200}}$$
for big enough $c$ (in fact, here the estimate $c-n \ge 10^{100}\sqrt{c\log(c)}$ would have been enough).

It remains to combine all our estimates and apply the min-max principle. We pick a subset $\mathcal{V} = \{ F_1, F_2,\ldots , F_n\}$ of $\mathcal{W}$ consisting of exactly $n$ elements and let $V$ be their linear span in $\Fo$. We want to show that $\mathcal{V}$ is linearly independent and that for any $F\in V\backslash\{0\}$ we have
$$\frac{\int_{\left(-\frac{\sqrt{c}}{2},\frac{\sqrt{c}}{2}\right)^2} |F(z)|^2e^{-\pi |z|^2}dz}{\int_\Cm |F(z)|^2e^{-\pi |z|^2}dz} \ge 1-\exp(-\kappa(\sqrt{c}-\sqrt{n})^2).$$
By subtracting this from $1$, this is the same as
$$\frac{\int_{\Cm\backslash\left(-\frac{\sqrt{c}}{2},\frac{\sqrt{c}}{2}\right)^2} |F(z)|^2e^{-\pi |z|^2}dz}{\int_\Cm |F(z)|^2e^{-\pi |z|^2}dz} \le \exp(-\kappa(\sqrt{c}-\sqrt{n})^2).$$
Let $F = \sum_{k=1}^n c_k F_k$ for some $c_k\in \Cm$. We have
$$\|F\|^2_{\Fo} = \sum_{k=1}^n\sum_{l=1}^n c_k\bar{c_l} \langle F_k, F_l\rangle = \sum_{k=1}^n |c_k|^2 + \sum_{k\neq l} c_k\bar{c_l}\langle F_k, F_l\rangle.$$
In absolute value each of these inner products is at most $\frac{1}{2c}$, so we get
$$\left|\|F\|^2_{\Fo}-\sum_{k=1}^n |c_k|^2\right|\le \frac{1}{2c}\left(\sum_{k=1}^n |c_k|\right)^2 \le \frac{n}{2c} \left(\sum_{k=1}^n |c_k|^2\right) \le \frac{1}{2}\left(\sum_{k=1}^n |c_k|^2\right),$$
where in the second step we used Cauchy--Schwarz inequality (or inequality between arithmetic mean and quadratic mean).  This means that $\|F\|_{\Fo}^2\ge \frac{1}{2}\sum_{k=1}^n |c_k|^2$, in particular that $F_k$ are linearly independent. Now, for the integral over $\Cm\backslash \left(-\frac{\sqrt{c}}{2},\frac{\sqrt{c}}{2}\right)^2$, again using the inequalty $|\sum_{k=1}^n b_k|^2 \le c\sum_{k=1}^n |b_k|^2$ pointwise, we get for $\eta = \frac{1}{200}$
\begin{align*}\int_{\Cm\backslash\left(-\frac{\sqrt{c}}{2},\frac{\sqrt{c}}{2}\right)^2} |F(z)|^2e^{-\pi |z|^2}dz \le c\sum_{k=1}^n |c_k|^2\int_{\Cm\backslash\left(-\frac{\sqrt{c}}{2},\frac{\sqrt{c}}{2}\right)^2} |F_k(z)|^2e^{-\pi |z|^2}dz\le\\ \frac{1}{2}\exp(-\eta(\sqrt{c}-\sqrt{n})^2)\sum_{k=1}^n |c_k|^2.
\end{align*}
Thus, the ratio
$$\frac{\int_{\Cm\backslash\left(-\frac{\sqrt{c}}{2},\frac{\sqrt{c}}{2}\right)^2} |F(z)|^2e^{-\pi |z|^2}dz}{\int_\Cm |F(z)|^2e^{-\pi |z|^2}dz}$$
is at most $\exp(-\eta(\sqrt{c}-\sqrt{n})^2)$, as required.
\begin{remark}Both Theorem \ref{upper coherent} and Theorem \ref{lower coherent} are applicable to many more domains than just squares, for example for scalings of compact piecewise $C^1$-domains (much weaker conditions are admissible). Theorem \ref{upper coherent} requires finding a subdomain such that it takes big enough portion of the measure and the difference of the domains has small diameter, which can be achieved by looking at the tubular neighboorhood of the boundary, for example. For Theorem \ref{lower coherent} we also need to find a suitable subdomain, and then almost cover it by disks, which can either be done by mimicking the proof of Lemma \ref{lem} in \cite{Lieb} or by considering first the Whitney decomposition of the subdomain and then filling each of the squares of the decomposition using Lemma \ref{lem}.
\end{remark}

\section{Paley--Wiener space}
We will be working with the space $PW_\pi$ defined as the set of functions $f\in L^2(\R)$ such that $\supp \hat{f}\subset [-\frac{1}{2},\frac{1}{2}]$. By the Fourier inversion, functions from $PW_\pi$ are entire functions in the whole complex plane and they satisfy the pointwise bound for $x, y\in\R$
$$|f(x+iy)|\le \|f\|_{L^2(\R)}e^{\pi |y|},$$
in particular $f(z)e^{i\pi z}$ is bounded in the upper half-plane. Additionally, since $\widehat{f'}(x) = 2\pi i x \hat{f}(x)$ we also have the following bound  for $x\in \R$
$$|f'(x)|\le \frac{2\pi}{\sqrt{12}}\|f\|_{L^2(\R)} \le 2\|f\|_{L^2(\R)}.$$

Just like in the Bargmann--Segal--Fock space, shifts map $PW_\pi$ to itself, and here we do not even have to add any correcting factors: if $z_0\in \Cm$ and $f\in PW_\pi$ then $f(\cdot - z_0)\in PW_\pi$. However, the equality of norms $\|f\|_{L^2(\R)} = \|f(\cdot - z_0)\|_{L^2(\R)}$ only holds for $z_0\in\R$ (and this is the only case that we will use).

Lastly, we will use the fact that if we multiply two or more functions then on the Fourier transform side this corresponds to a convolution which can not make the support too large. Specifically, if $\supp \widehat{f_k}\subset [-r_k, r_k]$ for $k = 1, \ldots , N$ then 
$$\supp \F(f_1f_2\ldots f_N)\subset \left[-\sum_{k=1}^n r_k, \sum_{k=1}^n r_k\right].$$
\section{Proof of Theorem \ref{lower time}}
We begin by covering the case $c-n \le c^{7/8}$ which follows from the combination of Theorem \ref{Slepian} and Theorem \ref{Karnik}. If $c-n \le 10^{100}\log(c)$ then Theorem \ref{Slepian} tells us that $\lambda_n(c)$ is bounded away from $0$, hence by decreasing $\eta$ if needed be we can make sure that the estimate holds in this regime. For $10^{100}\log(c) \le c - n \le c^{7/8}$ we first of all notice that $\log(\frac{2c}{c-n})$ is at least $\frac{1}{8}\log(c)$, so the factor is still logarithmic in $c$. By Theorem \ref{Slepian} we know that for $\eps < \frac{1}{3}$, say, for big enough $c$ the point $n = [c]$ is inside the region $\{n: \eps < \lambda_n(c) < 1-\eps\}$, and since this set is clearly a set of consecutive integers (since $\lambda_n(c)$ is decreasing in $n$), and $n_c = [c]$ is in this region, if we consider $n = [c] - |\{n:\eps < \lambda_n(c) < 1-\eps\}|$ it will no longer be in this set, hence $\lambda_n(c) \ge 1-\eps$. By using the estimate for $|\{n:\eps < \lambda_n(c) < 1-\eps\}|$ from Theorem \ref{Karnik} for $\eps = \exp(-\frac{9\eta}{\log(c)}(c-n))$ we get the desired bound if $\eta$ is small enough in this regime. So, from now on we will assume that $c-n > c^{7/8}$. Lastly, we will also assume that $n \ge 0.99 c$ because otherwise we can just use monotonicity of $\lambda_n(c)$ and the case $n = [0.99c]+1$ like we did in the proof of Theorem \ref{lower coherent} since the bound is always exponential in $c$ (this regime is the only reason we need $2$ in $\log(\frac{2c}{c-n})$ so that the logarithm is at least $\log(2)>0$).

Since the eigenvalues depend only on the product of lengths of the intervals, we will from now on assume that $I = [-\frac{1}{2},\frac{1}{2}]$ and $J = [-\frac{c}{2},\frac{c}{2}]$. Put $r = 0.01 \frac{c-n}{\log(\frac{2c}{c-n})}$ and let us consider the Whitney decomposition of $J$ up to the size $r$: we set $J_k = [\frac{c}{2}-\frac{c}{2^{k}},\frac{c}{2}-\frac{c}{2^{k+1}}] = [L_k, R_k]$, $J_{-k} =-J_k$ for $1\le k\le l$ where $l$ is such that $\frac{c}{2}-r\in J_l$ so that $2^{l}r \le c \le 2^{l+1}r$. For each $k = 1, 2,\ldots , l$ put $\gamma_k = \left(1 - \frac{r2^{k-10}}{c}\right)^{-1} > 1$ and let $A_k = [L_k + c^{6/8}, R_k-c^{6/8}]\cap \gamma_k\Z$, $A_{-k} = -A_k$. First, we will count the total number of points in $\cup_{k=1}^l A_k$. We have
$$\sum_{k=1}^l |A_k| \ge \sum_{k=1}^l \left(\frac{c}{\gamma_k 2^{k+1}} - (2c^{6/8}+2) \right)= -l(2c^{6/8}+2)+\sum_{k=1}^l \frac{c}{2^{k+1}}\left(1-\frac{r2^{k-10}}{c}\right)=$$
$$\frac{c}{2} - \frac{c}{2^{l+1}}-l(2c^{6/8}+2+r2^{-11}) \ge \frac{c}{2}-r-\frac{lr}{200}\ge \frac{c}{2}-\frac{lr}{100},$$
%\end{align*}
where the first inequality uses that $|W\cap \gamma_k\Z|$ for an interval $W$ is at least $\frac{|W|}{\gamma_k}-2$ and the last two steps use that $r\ge 10^{10}c^{6/8}$ and $l\ge 200$ for big enough $c$ (here we use the assumption $c-n\ge c^{7/8}$).

We have $l \le 2\log(\frac{c}{r})$. Therefore $\frac{lr}{100c} \le \frac{1}{50}\frac{r}{c}\log(\frac{c}{r})$. Let $x = \frac{r}{c} < y = \frac{c-n}{c}< 0.1$. We know that 
$$x = 0.01\frac{y}{\log(\frac{2}{y})}.$$

Since $y$ is small, it is known that we can estimate $y\ge 10x\log(\frac{1}{x})\ge 20xl$. Thus, $$\frac{lr}{100} = \frac{cxl}{100}\le \frac{cy}{2}=\frac{c-n}{2}.$$

Therefore, the total number of points in $A=\bigcup_{k=1}^l A_k\cup A_{-k}$ is at least $2(\frac{c}{2}-\frac{c-n}{2}) = n$. Our goal is to associate to each $p\in A$ a function $f_p\in PW_\pi$ such that for all $p, q\in A$ we have $f_p(q)\approx \delta_{p,q}$ and all functions $f_p$ have strong concentration to the interval $[-\frac{c}{2},\frac{c}{2}]$. For $p\in A_k\cup A_{-k}$ the function $f_p$ will have the form
$$f_p(x) = \frac{\sin(\frac{\pi (x-p)}{\gamma_k})}{\frac{\pi (x-p)}{\gamma_k}}g_k(x-p)h_k(x-p),$$
where $h_k$ and $g_k$ are functions to be chosen such that $h_k(0)=g_k(0)=1$. The sine function is responsible for the exact biorthogonality property for $q\in A_k\cup A_{-k}$ so that $f_p(q)=\delta_{p,q}$. Function $g_k$ is responsible for almost biorthogonality for $q\in A\backslash (A_k\cup A_{-k})$. This will be done by making it be localized to the $c^{6/8}$ scale (this is the reason we cut this much from both sides of intervals $J_k$). However, since we will have a huge margin here, the function $g_k$ is by far the least important of the three. Finally, function $h_k$ will be responsible for the localization of $f_p$ to the interval $[-\frac{c}{2},\frac{c}{2}]$. Because the distance from any $p\in A_{k}\cup A_{-k}$ to $\R\backslash[-\frac{c}{2},\frac{c}{2}]$ is proportional to $\frac{c}{2^k}$ we will make $h_k$ be localized on this scale.

The most direct way to construct a function which is concentrated on some scale is to, ironically, take the first eigenfunction of the time-frequency localization operator and apply the estimates of Fuchs \cite{Fuc}. However, to keep this part of the text self-contained we will construct such a well-concentrated function explicitly.
\begin{proposition}\label{function prop}
Let $T, W>0$ be such that $TW \ge 1$. The function $$f_{T, W} =\F \left(e^{-\pi x^2 \frac{W}{T}}\chi_{[-\frac{T}{2},\frac{T}{2}]}(x)\right)$$ satisfies
$$\int_{|t| >\frac{W}{2}} |f_{T,W}(t)|^2dt \le \sqrt{\frac{T}{W}}10e^{-\frac{\pi TW}{2}},$$
$$\int_{\R} |f_{T,W}(t)|^2dt\le \sqrt{\frac{T}{W}},$$
$$f_{T, W}(0)\ge \frac{1}{10}\sqrt{\frac{T}{W}},$$
$$|f_{T,W}(t)|\le f_{T,W}(0), t\in\R.$$
\end{proposition}
\begin{proof}
First of all we notice that all of the estimates respect the linear change of variables $t\to qt$, so we can without loss of generality assume that $T = W\ge 1$. In this case, since $e^{-\pi x^2}$ is its own Fourier transform, we get
 $$f_{T,T} = \F \left(e^{-\pi x^2} \chi_{[-\frac{T}{2},\frac{T}{2}]}\right) = e^{-\pi x^2} - \F \left(e^{-\pi x^2} \chi_{\R\backslash [-\frac{T}{2},\frac{T}{2}]}\right).$$
This means, using the triangle inequality and the Plancherel theorem, that 
$$\int_{|t| > \frac{W}{2}} |f_{T, T}(t)|^2 dt \le 8\int_{\frac{T}{2}}^\infty e^{-2\pi x^2}dx.$$
This is the tail of the Gaussian for which very precise estimates are known, but for us a simple bound is sufficient:
$$\int_{\frac{T}{2}}^\infty e^{-2\pi x^2}dx \le \frac{2}{T}\int_{\frac{T}{2}}^\infty xe^{-2 \pi x^2}dx = \frac{1}{2\pi T}e^{-\frac{\pi T^2}{2}} \le e^{-\frac{\pi T^2}{2}},$$
where in the last step we used $T \ge 1$.

The second estimate is the easiest, as it follows directly from the Plancherel theorem:
$$\int_\R |f_{T,T}(t)|^2dt = \int_{-\frac{T}{2}}^{\frac{T}{2}} e^{-2\pi x^2}dx \le \int_\R e^{-2\pi x^2}dx =\frac{1}{\sqrt{2}} < 1.$$

For the value at $0$ we have
$$f_{T,T}(0) = \int_{-\frac{T}{2}}^{\frac{T}{2}}e^{-\pi x^2}dx \ge \int_{-\frac{1}{2}}^{\frac{1}{2}}e^{-\pi x^2}dx \ge e^{-\frac{\pi}{4}} \ge \frac{1}{e} > \frac{1}{10}.$$
Finally, the estimate $|f_{T,T}(t)|\le f_{T,T}(0)$ follows from the fact that the Fourier transform of $f_{T,T}$ is non-negative.
\end{proof}
\begin{remark}
Note that the exponential $e^{-\frac{\pi TW}{2}}$ corresponds to $e^{-\frac{\pi}{2}c}$, which is twice worse than the estimate of Fuchs which gives $e^{-\pi c}$.
\end{remark}
The last ingredient that we need is a way to get from the $L^2$-bound to the pointwise bound. Although for our explicit functions we can obtain explicit pointwise bounds, we will instead use the following lemma which interpolates between $L^2$-bound and $L^\infty$-bound for the derivative, in particular because it will also be used in the proof of the upper bound for the eigenvalues. 
\begin{lemma}\label{Kolm}
Let $\Gamma\subset \R$ be a ray (meaning one of the sets $[x_0,+\infty)$, $(x_0,+\infty)$, $(-\infty,x_0)$ or $(-\infty, x_0]$) and let $f:\Gamma\to\Cm$ be a $C^1$-function which is in $L^2(\Gamma)$. Then for all $t\in \Gamma$ we have
$$|f(t)|\le 3^{1/3} \|f\|_{L^2(\Gamma)}^{2/3}\sup_{x\in\Gamma} |f'(x)|^{1/3}.$$
\end{lemma}
\begin{proof}
We will consider only the case $\Gamma = [x_0,+\infty)$, the other three cases are completely analogous. Let $C = \sup_{x\in\Gamma} |f'(x)|$ and consider the interval $U=\left[t, t+\frac{|f(t)|}{C}\right]\subset\Gamma$. For $x\in U$ we have by the fundamental theorem of calculus and the triangle inequality that $$|f(x)|\ge |f(t)| - C(x-t).$$
Squaring this estimate (note that the right-hand side is non-negative on $U$) and integrating it over $U$ we get
$$\int_\Gamma |f(x)|^2dx \ge \int_U |f(x)|^2dx \ge \int_t^{t+\frac{|f(t)|}{C}} (|f(t)|-C(x-t))^2dx=\frac{|f(t)|^3}{3C},$$
which gives the desired estimate after multiplying by $3C$ and taking the cube root.
\end{proof}
\begin{remark}This lemma is in fact a special case of the Landau--Kolmogorov inequality. More general version of this inequality was used by the author in \cite{KulikovFI} where the estimates on the eigenvalues $\lambda_n(c)$ were applied to the questions of Fourier interpolation.
\end{remark}
Now, we are finally in position to choose functions $g_k$ and $h_k$. Put $\delta_k = \frac{1-\gamma_k^{-1}}{2} = \frac{r2^{k-11}}{c}$. Note that $\delta_k \ge \delta_1 \ge c^{-2/8}$ for big enough $c$. We set $$g_k(x) = \frac{f_{\delta_k, c^{6/8}}(x)}{f_{\delta_k, c^{6/8}}(0)},\quad h_k = \frac{f_{\delta_k, \frac{c}{2^k}}(x)}{f_{\delta_k, \frac{c}{2^k}}(0)}.$$
First of all we note that $g_k(0) = h_k(0) = 1$ and $\supp \widehat{g}_k = \supp \widehat{h}_k = [-\frac{\delta_k}{4},\frac{\delta_k}{4}]$, so that the support of the Fourier transform of $f_p$ is contained in $[-\frac{1}{2},\frac{1}{2}]$. Observe also that $\delta_k \frac{c}{2^k} \ge \delta_k r \ge \delta_k c^{6/8} \ge c^{1/2}$, so for big enough $c$ we can apply Proposition \ref{function prop} to $g_k$ and $h_k$.

We begin with estimation of $f_p(q)$ for $q\in A\backslash (A_k\cup A_{-k})$. The first factor and the third factor are at most $1$ in absolute value so we will ignore them. For the second factor, since $|q-p|\ge c^{6/8}$, we can apply Proposition \ref{function prop} and Lemma \ref{Kolm} to $g_k(q-p)$ to get
$$|g_k(q-p)| \le \frac{1}{f_{\delta_k, c^{6/8}}(0)}3^{1/3}\left(\sqrt{\frac{\delta_k}{c^{6/8}}}10e^{-\frac{\pi \delta_k c^{6/8}}{2}}\right)^{2/3}\sup_{x\in \R}|f'_{\delta_k, c^{6/8}}(x)|^{1/3}.$$
The derivative we can estimate from above by crudely saying that the Fourier transform of $f_{\delta_k, c^{6/8}}$ is contained in $[-\frac{1}{2},\frac{1}{2}]$, so we can estimate it through the $L^2$-norm of the function while the value at $0$ we estimate from below using Proposition \ref{function prop}. Combining everything with bounds $1 \ge \delta_k \ge c^{-2/8}$ we can see that all the terms except the exponential will together give us at most $c^{10}$ for big enough $c$. So, we get
$$|g_k(q-p)|\le c^{10} e^{-\frac{\pi \delta_k c^{6/8}}{3} }\le c^{10}e^{-\delta_kc^{6/8}} \le c^{10}e^{-c^{1/2}} \le c^{-10}$$
for big enough $c$.

Finally, we need to estimate the $L^2$-norm of $f_p$ outside of $[-\frac{c}{2},\frac{c}{2}]$. For this we instead estimate the first two factors in the definition of $f_p$ by $1$ and only focus on the third factor. We notice that the distance from $A_k\cup A_{-k}$ to $\R\backslash [-\frac{c}{2},\frac{c}{2}]$ is at least $\frac{c}{2^{k+1}}$, which is exactly the radius to which the function $f_{\delta_k, \frac{c}{2^k}}$ is concentrated. So, we can apply Proposition \ref{function prop} and get
$$\int_{\R\backslash [-\frac{c}{2},\frac{c}{2}]} |f_p(t)|^2dt \le 100 \frac{\frac{c}{2^k}}{\delta_k} \sqrt{\frac{\delta_k}{\frac{c}{2^k}}} 10 e^{-\frac{\pi \delta_k\frac{c}{2^k}}{2}}\le c^{10}e^{-\frac{r}{2^{11}}}\le c^{-10}e^{-\frac{r}{10000}},$$
where we crucially used that $\delta_k\frac{c}{2^k} = \frac{r}{2^{11}}$ and that $r\ge 10^{100}\log(c)$ (which corresponds to $c-n\ge 10^{100}\log^2(c)$).

It remains to apply the min-max principle to the linear span of our functions. Let us pick $B=\{p_1, p_2, \ldots , p_n\}\subset A$ of size $n$ and let $V$ be the linear span of $f_{p_k}, k = 1,\ldots , n$. Let us take any linear combination $f = \sum_{k=1}^n c_kf_{p_k}$ and let $|c_m|$ be the biggest among all $|c_k|$. Then 
$$|f(p_m)| = |c_m - \sum_{k\neq m} c_k f_{p_k}(p_m)| \ge |c_m|(1 - \sum_{k\neq m} |f_{p_k}(p_m)|)\ge |c_m|(1-c c^{-10}) \ge \frac{|c_m|}{2},$$
therefore by the pointwise bound $\|f\|_{L^2(\R)}\ge \frac{|c_m|}{2}$, in particular $\dim V = n$. Next, just like in the proof of Theorem \ref{lower coherent} we wish to estimate the ratio
$$\frac{\int_{\R\backslash[-\frac{c}{2},\frac{c}{2}]} |f(t)|^2dt}{\|f\|^2_{L^2(\R)}}.$$
We already estimated the denominator from below. For the numerator we simply use triangle inequality to get that it is at most
$$n\sum_{k=1}^n |c_k|^2\int_{\R\backslash[-\frac{c}{2},\frac{c}{2}]} |f_{p_k}(t)|^2dt \le c^{-8}e^{-\frac{r}{10000}}\max_{k} |c_k|^2.$$
Combining everything, we get
$$\frac{\int_{\R\backslash[-\frac{c}{2},\frac{c}{2}]} |f(t)|^2dt}{\|f\|^2_{L^2(\R)}} \le e^{-\frac{r}{10000}},$$
so $\lambda_n(c) \ge 1 - e^{-\frac{r}{10000}}$. It remains to recall that $r = 0.01 \frac{c-n}{\log(\frac{2c}{c-n})}$ so we get the desired estimate.
\begin{remark}
The only place where we needed a strong bound like $c-n\ge c^{7/8}$ is in the estimation of $f_p(q)$ for $p\neq q$, the concentration part only required $c-n\ge 10^{100}\log^2(c)$. By more careful analysis this part will work if $c-n \ge c^{1/2 + \eps}$ but we were unable to break this barrier with a simple explicit construction. However, using the theory of entire functions of exponential type we can replace the first factor in the definition of $f_p$  so that $f_p(q) = \delta_{p, q}$ not only for $q\in A_k$ but also for $q\in A_{k-1}$ and $q\in A_{k+1}$, and for other $q$'s it can be checked that $c-n\ge 10^{100}\log^2(c)$ is enough. However, since we will anyway have to cover the $0 < c-n < 10^{100}\log^2(c)$ regime with Theorem \ref{Slepian} and Theorem \ref{Karnik} we decided to use them for a larger range of $n$'s so that for the remaining $n$'s our construction is completely explicit.
\end{remark}
\begin{remark}
We can make a similar argument to prove Theorem \ref{lower coherent} (with a slightly worse constant than $0.99$) by putting onto each disk shifts of partial products of the Weierstrass sigma function divided by one zero instead of shifts of monomials, and then look at the values of linear combinations at the corresponding points, but this requires way more work and gives a worse result so we do not present it.
\end{remark}
\section{Proof of Theorem \ref{upper time}}\label{UT}
First of all we will once again assume that $I = [-\frac{1}{2},\frac{1}{2}]$ and $J = [-\frac{c}{2},\frac{c}{2}]$ by scaling. We begin by noting that the eigenfunctions of $S_{I,J}$ with positive eigenvalues are necessarily supported on $I$. In particular, their Fourier transforms are defined pointwise. We are going to take a normalized linear combination $f=\sum_{k=1}^n c_k f_k$ of the eigenfunctions such that $\hat{f}(z) = 0, z\in A$ for some set $A\subset \R$ of size at most $n-1$. This set $A$ will be similar to the set constructed in the previous section, but instead of taking all the constants being small and positive we will instead make then big and positive, so that the total number of points in $A$ is slightly less than $n$ and not more. Lastly, we will for now assume that $c-n \le 10^{-1000}c$ and will cover the remaining $n$'s at the end of the proof.

Put $r = 10^{100}\frac{c-n}{\log(\frac{2c}{c-n})}$ and consider the Whitney decomposition of $J$ up to the size r:  we set $J_k = [\frac{c}{2}-\frac{c}{2^{k}},\frac{c}{2}-\frac{c}{2^{k+1}}] = [L_k, R_k]$, $J_{-k} =-J_k$ for $1\le k\le l$ where $l$ is such that $\frac{c}{2}-r\in J_l$ so that $2^{l}r \le c \le 2^{l+1}r$. By the above assumption $l\ge 10$. For each $k = 1, 2,\ldots , l$ put $\gamma_k = \left(1-\frac{r2^{k-10}}{c}\right)^{-1}>1$ and let $A_k = [L_k, R_k]\cap \gamma_k\Z$ and $A_{-k} = -A_k$. Note that we are no longer cutting away segments of size $c^{6/8}$ from the ends of $[L_k,R_k]$ because otherwise we will not get strong enough concentration on the borders between two such intervals. We have
$$\sum_{k=1}^l |A_k| \le \sum_{k=1}^l \frac{c}{\gamma_k 2^{k+1}} + 2 = 2l + \sum_{k=1}^l \frac{c}{2^{k+1}}\left(1-\frac{r2^{k-10}}{c}\right) = 2l+\frac{c}{2} - \frac{c}{2^{l+1}}-\frac{lr}{2^{11}},$$
where in the first step we used that $|A_k| \le \frac{R_k-L_k}{\gamma_k}+2$. Therefore, for $A = \bigcup_{k=1}^l A_k\cup A_{-k}$ we have
$$|A| < 4l + c - \frac{lr}{2^{10}}.$$

We want to show that $|A| < n$. For this it is enough to establish
$$\frac{lr}{2^{10}}\ge 4l + (c-n).$$

Recall that $c-n \ge \log^2(c)$. Therefore in particular $r \ge 10^{100} \log(c)\ge 10^{100}$. Hence, the term $4l$ on the right-hand side is easily covered by a half of the term on the left-hand side. For the $c-n$ term we want to show that
$$\frac{10^{100}}{2^{11}}\frac{l(c-n)}{\log(\frac{2c}{c-n})}\ge c-n$$
which is equivalent to
$$\frac{10^{100}}{2^{11}}l \ge \log\left(\frac{2c}{c-n}\right).$$
We have $$2l\ge l +1 \ge \frac{1}{2}\log\left(\frac{c}{r}\right) = \frac{1}{2}\log\left(\frac{\frac{c}{c-n}}{\log(\frac{2c}{c-n})}\right).$$

As in the previous section, put $y = \frac{c-n}{c} < 10^{-1000}$. Then the desired inequality takes form
$$\frac{10^{100}}{2^{13}}\log\left(\frac{1}y{\log\left(\frac{2}{y}\right)}\right) \ge \log\left(\frac{2}{y}\right)$$
which holds for $y < 0.01$.

So, $|A| < n$ and we can take a linear combination $f$ of the first $n$ eigenfunctions for which $\|f\|_{L^2(\R)}=1$ and $\hat{f}(z)=0,z\in A$. For brevity, denote $g=\hat{f}\in PW_\pi$. By the min-max principle we have
$$\int_{|x| > \frac{c}{2}}|g(x)|^2dx \le 1-\lambda_n(c).$$

Therefore, by Lemma \ref{Kolm} and the estimate for the derivative of the function $g$ on $\R$ we get
$$|g(x)|\le 10(1-\lambda_n(c))^{1/3}, |x| > \frac{c}{2}.$$

Our next goal is to obtain from this an estimate on $|g(x+iy)|$ for $x\in [-\frac{c}{2},\frac{c}{2}]$ whenever $\min\left(|x-\frac{c}{2}|,|x+\frac{c}{2}|\right)< 10|y|$ which is better than the uniform estimate $e^{\pi |y|}$ that is valid for all $x,y\in\R$. We will only consider the case $y > 0, x > 0$, the other three cases will be completely analogous. 

Recall that $g(z)e^{i\pi z}$ is a bounded analytic function in the upper half-plane, therefore the logarithm of its absolute value is subharmonic there and the value at any point $x+iy, y > 0$ is bounded from above by the Poisson integral:
$$\log |g(x+iy)e^{-\pi y}|\le \frac{1}{\pi}\int_\R \frac{y}{(t-x)^2+y^2}\log |g(t)|dt.$$

For $t \le \frac{c}{2}$ we estimate $\log |g(t)|$ from above by $0$ and for $t > \frac{c}{2}$ we estimate $\log |g(t)|$ from above by our pointwise bound and get
$$\log |g(x+iy)| \le \pi y + \log(10(1-\lambda_n(c)^{1/3})\frac{1}{\pi}\int_{\frac{c}{2}}^\infty \frac{y}{(t-x)^2+y^2}dt.$$

For $ |x-\frac{c}{2}| < 10 |y|$ it is not hard to see that the integral is at least $10^{-10}$ by looking only at $t$ with $\frac{c}{2}\le t \le \frac{c}{2} + y$. Therefore, we get the following estimate on $|g(x+iy)|$ for $x\in [-\frac{c}{2},\frac{c}{2}]$ and $\min\left(|x-\frac{c}{2}|,|x+\frac{c}{2}|\right)< 10|y|$:
$$\log |g(x+iy)| \le \pi |y| + 10^{-10}\log(10(1-\lambda_n(c))^{1/3}).$$

Now, we want to use this bound together with the knowledge that $g(z) = 0, z\in A$ to obtain an estimate for $g(x), x\in [-\frac{c}{2},\frac{c}{2}]$. There are two cases to consider: $x\in J_k\cup J_{-k}$ for some $1\le k < l$ or $x\in [\frac{c}{2}-\frac{c}{2^{l}}, \frac{c}{2}]\cup [-\frac{c}{2}, -\frac{c}{2}+\frac{c}{2^{l}}]$. We begin with the easier second case, for which the fact about the zeroes is not needed. It will be similar to an argument used in the proof of Theorem \ref{upper coherent}.

Let $x\in  [\frac{c}{2}-\frac{c}{2^{l}}, \frac{c}{2}]$ (the other case is completely analogous). Consider a circle with centre $x$ and radius $r$. Since $\log |g(z)|$ is subharmonic, the integral of it over the circle is at least $\log |g(x)|$:
$$\log |g(x)| \le \int_0^1 \log |g(x + re^{2\pi i t})|dt.$$

Since $|\frac{c}{2}-x| \le 4r$, for the top and bottom quarters of the circle our stronger pointwise bound on $|g(a+ib)|$ applies, while for other points we will use a bound $|g(a+ib)| \le e^{\pi |b|}$. Plugging this in we get
\begin{align*}\log |g(x)| \le \frac{1}{2} 10^{-10} \log(10(1-\lambda_n(c))^{1/3} + \int_0^1 \pi r |\sin(2\pi t)| dt\\ =\frac{1}{2}10^{-10}\log(10(1-\lambda_n(c))^{1/3} +2r.
\end{align*}

If we assume that $\lambda_n(c) \ge 1 - \exp\left(-\kappa\frac{c-n}{\log(\frac{2c}{c-n})}\right)$ for big enough $\kappa$ then this estimate implies in particular that $\log |g(x)| \le -r$ (recall that we know that $r$ is big enough, so we can absorb all the constants).

Next, we turn to the case $x\in J_k\cup J_{-k}$ for some $1 \le k < l$. Once again, we will only do the case $x\in J_k = [L_k, R_k]$, the other being completely analogous. We consider a circle with centre $x$ and radius $R=\frac{c}{2^{k+2}}$. Here, just subharmonicty of $\log |g(z)|$ would be insufficient, we will use the full power of Jensen's formula:
$$\log |g(x)| = \int_0^1 \log |g(x+Re^{2\pi i t})|dt + \sum_{|x-z_m| < R} \log \frac{|x-z_m|}{R},$$
where the sum is over the zeroes $z_m$ of $g$.	For the integral term by doing the same computation as in the previous case we get
$$\int_0^1 \log |g(x+Re^{2\pi i t})|dt \le 2R + \frac{1}{2}10^{-10}\log(10(1-\lambda_n(c))).$$

For the sum over the zeroes we will leave only $z_m\in A$, this clearly can make the sum only bigger. We will also remove the zeroes for which $|x-z_m| < 1$. Note that the remaining zeroes lie in $A_{k-1}\cup A_k\cup A_{k+1}$, and the spacing between them is at most $\gamma_{k+1} < 2$ except possibly for four zeroes on the boundaries of the intervals $J_{k-1}, J_k, J_{k+1}$ (if $k = 1$ then instead of $A_0, J_0$ we have to consider $A_{-1}, J_{-1}$, but the same reasoning holds). Each of these four missing points costs us at most $\log(R)\le \log(c)$. The remaining zeroes we split into the ones bigger and smaller than $x$ and estimate each of the sums by the integral of $\log(\frac{t}{R})$ from $0$ to $R-2\gamma_{k+1}$ since the function $\log(s)$ is increasing. Moreover, adding at most four extra points (each of which costs us at most an extra $\log(c)$) we can also take the integral all the way up to $R$. We get
\begin{align*}\sum_{|x-z_k| < R} \log \frac{|x-z_k|}{R} \le 8\log(c)+\frac{2}{\gamma_{k+1}}\int_0^{R}\log\left(\frac{t}{R}\right)dt=8\log(c)-\frac{2R}{\gamma_{k+1}}\\=8\log(c)-2R+\frac{Rr2^{k-9}}{c}=8\log(c)-2R+\frac{r}{2^{11}}.
\end{align*}

Combining everything we get
$$\log |g(x)| \le 8\log(c) + \frac{1}{2}10^{-10}\log(10(1-\lambda_n(c))) + \frac{r}{2^{11}}.$$

Just like before, if we assume that $\lambda_n(c) \ge 1 - \exp\left(-\kappa\frac{c-n}{\log(\frac{2c}{c-n})}\right)$ for big enough $\kappa$ and $c -n \ge \log^2(c)$ so that $r \ge \frac{\log(c)}{4}$ then this estimate gives us that $\log |g(x)| \le -r$.

In the case $c -n\ge 10^{-1000}c$ which we left out in the beginning we can ignore any zeroes of the function $g$ and just do the integral computation for all $x\in [-\frac{c}{2},\frac{c}{2}]$ by taking the circle with centre $x$ and radius $\frac{c}{2}$ and get the same bound $\log |g(x)|\le -r$ if $\lambda_n(c) \ge 1 - \exp\left(-\kappa\frac{c-n}{\log(\frac{2c}{c-n})}\right)$ for big enough $\kappa$.

To arrive at a contradiction we note that if $\lambda_n(c) \ge 1 - \exp\left(-\kappa\frac{c-n}{\log(\frac{2c}{c-n})}\right)$ for big enough $\kappa$ then clearly $\lambda_n(c) \ge \frac{1}{2}$ if $c-n\ge \log^2(c)$ and $c$ is big enough. Therefore,
$$\int_{-\frac{c}{2}}^{\frac{c}{2}}|g(x)|^2dx \ge \frac{1}{2}.$$
On the other hand, from our estimate we get
$$\int_{-\frac{c}{2}}^{\frac{c}{2}}|g(x)|^2dx \le \int_{-\frac{c}{2}}^{\frac{c}{2}}e^{-2r} \le ce^{-2r}$$
which is less than $\frac{1}{2}$ since $c-n\ge \log^2(c)$ and $r\ge 10\frac{c-n}{\log(c)}$.
\begin{remark}\label{UC}
We can try doing a similar argument in the coherent state transform case by constructing an appropriate subset of the plane and looking at the linear combination of eigenfunctions vanishing at all of them. This approach actually gives the correct order $1-\exp(-\kappa(\sqrt{c}-\sqrt{n})^2)$ but for a worse region of $n$, on the order of $c-n\ge c^{0.9}$. Because of this and since this argument is way more technical than the proof of Theorem \ref{upper coherent} we do not present it.
\end{remark}
\section{Concluding remarks}
In this work we established the asymptotic behaviour for the quantities $1-\lambda_n(c)$ and $1-\mu_n(c)$ for $n < c$. Of course, this still leaves many things to be done. Probably the most ambitious one is to obtain integral approximations to $\log(1-\lambda_n(c))$ and $\log(1-\mu_n(c))$ for $n < c$ with a small error term in the spirit of Theorem \ref{Bonami}. In particular, it would be interesting to find for $0 < \eps < 1$ the limits
$$\gamma_\eps = -\lim_{c\to\infty} \frac{\log(1-\lambda_{(1-\eps)c}(c))}{c}, \delta_\eps =-\lim_{c\to\infty} \frac{\log(1-\mu_{(1-\eps)c}(c))}{c}.$$
From the present work we know that, if they exist, for small $\eps > 0$ they should behave like $\gamma_\eps \sim \frac{\eps}{\log(1/\eps)}$ and $\mu_\eps \sim \eps^2$, but their precise values likely involve some integrals with special functions.

Another interesting problem is to extend the range of $n$ for which our results apply. The natural boundaries are $c-n\sim \log(c)$ in the time-frequency localization case and $c-n\sim \sqrt{c}$ in the coherent state transform case, when the eigenvalues become bounded away from both $1$ and $0$. However, we were able to get this close only in Theorem \ref{lower time} with the help of Theorem \ref{Karnik}, in all other results we have to at the very least compensate for the polynomial in $c$ losses, leading to worse ranges in $n$.

Lastly, let us say a few words about what happens in the case $n > c$. In the time-frequency localization case extremely sharp estimates are provided by Theorem \ref{Slepian}, Theorem \ref{Karnik} and Theorem \ref{Bonami}, but it is worth mentioning that the methods of Section~\ref{UT} can give an upper bound on the eigenvalues for $c+10^{100}\log^2(c) < n < 10^{100}c$ which has the same asymptotic behaviour as $\log(\lambda_n(c))$ that we get in Theorem \ref{Bonami}. 

In the coherent state transform case these methods can also give an upper bound of expected magnitude for the region $c + c^{0.9} < n < 10c$, see Remark \ref{UC} (this can be extended to all $n \ge 10c$ as well by comparing with the case of a disk). For the lower bound, acting similarly to Section \ref{UPC}, but instead taking a bigger square and using its eigenfunctions for the min-max procedure, we can obtain lower bounds on $\mu_n(c)$ of expected magnitude for $n > c + \sqrt{c\log(c)}$. The details of these computations will appear elsewhere.
%{\color{red} in conclusion explicit constants, asymptotic instead of bounds (at least $n = (1-\delta)c$), $n > c$ case + covering by the disk (in the introduction or somewhere mention the decay for the disk), zeroes of eigenfunctions?, breaking the $\log^2$ barrier}
\subsection*{Acknowledgments} 
I am grateful to Fedor Nazarov for his help with the time-frequency localization operator. I also would like to thank Aleksandr Logunov for discussions on harmonic measure and logarithmic integrals, Jan Philip Solovej for telling me about the paper \cite{Lieb} and Mikhail Sodin for helpful comments about Jensen's formula. This work was supported by BSF Grant 2020019, ISF Grant 1288/21, and by The Raymond and Beverly Sackler Post-Doctoral Scholarship and by the VILLUM Centre of Excellence for the Mathematics of Quantum Theory (QMATH) with Grant No.10059.

\end{document}